\documentclass{article}%
\usepackage{amsmath}
\usepackage{amsfonts}
\usepackage{amssymb}
\usepackage{graphicx}
\usepackage{color}%
\providecommand{\U}[1]{\protect\rule{.1in}{.1in}}
\newtheorem{thm}{Theorem}

\newtheorem{cor}[thm]{Corollary}

\newtheorem{example}[thm]{Example}

\newtheorem{lem}[thm]{Lemma}

\newtheorem{rem}[thm]{Remark}

\newenvironment{proof}[1][Proof]{\noindent\textbf{#1.} }{\ \rule{0.5em}{0.5em}}
\begin{document}

\title{{Optimal stopping of McKean-Vlasov diffusions via regression on particle
systems}}
\author{Denis Belomestny\thanks{{\footnotesize Duisburg-Essen University. Email:
\texttt{denis.belomestny@uni-due.de}}} \ and John
Schoenmakers\thanks{{\footnotesize Weierstrass Institute of Applied
Mathematics. Email: \texttt{schoenma@wias-berlin.de}}}\\\vspace{20pt} }
\maketitle

\begin{abstract}
In this paper we study  optimal stopping problems
for nonlinear Markov processes driven by a McKean-Vlasov SDE and  aim at  solving  them numerically by
Monte Carlo. To this end we propose a novel regression  algorithm based on the
corresponding particle system and prove its convergence. The proof of
convergence is based on  perturbation analysis of a related linear regression
problem. The performance of the proposed algorithms is illustrated by a
numerical example.

\end{abstract}

\section{Introduction}

Numerical solution of multidimensional optimal stopping problems remains an
important and active area of research with applications in finance, operations research
and control. As the underlying models are getting more and more complex, the
computational issues are becoming more relevant than ever. As a matter of
fact, analytic and usual finite difference methods for solving optimal
stopping problems deteriorate in high-dimensional problems. Therefore
attention has turned to probabilistic approaches, based on Monte Carlo based
approximative dynamic programming. Historically, one of the first motivating
examples was the pricing of an American (or Bermudan) put option in a
Black-Scholes model. Not much later, many American options that showed up
involved high dimensional underlying processes, which has led to the
development of several Monte Carlo based methods in the last decades (see e.g.
\cite{Gl}). Pricing American derivatives, hence solving optimal stopping
problems, via Monte Carlo has always been viewed as a challenging task,
because it requires backward dynamic programming that seems to be incompatible
with the forward structure of Monte Carlo methods. In particular much research
was focused on the development of fast methods to compute approximations to
the optimal exercise policy. The seminal paper of Longstaff and
Schwartz~\cite{J_LS2001} proposed to use a cross-sectional regression over a
Monte Carlo sample to compute the conditional expectations involved in the
dynamic programming algorithm. The key innovation in \cite{J_LS2001} was the
use of the estimated conditional expectations for decision-making, rather than
quantification of expected gains. This corresponds to replacing a regression
problem with a classification one, see \cite{egloff2005monte}. Originally
designed for the setup of American option pricing, this algorithm, which we
term RMC (Regression Monte Carlo), has become widely accepted in financial
mathematics, insurance and dynamic programming settings. It has also been
implemented in many proprietary valuation systems employed by the financial
industry. The great success of RMC is due to its flexible and simple
implementation as well as its strong empirical performance. Other eminent
examples of fast approximation methods include the functional optimization
approach of \cite{A}, the mesh method of \cite{BG}, the regression-based
approaches of \cite{Car}, \cite{J_TV2001}, \cite{egloff2005monte} and
\cite{B1}.
\par
In this paper we propose and study a simulation based method for solving an
optimal stopping problem for nonlinear Markov processes of the McKean-Vlasov type.
In spite of  extensive literature on stochastic particle systems
corresponding to MV-SDEs, generic numerical methods for solving optimal
stopping problems in this context are hard to find (to the best of our
knowledge). Our study is motivated by the recent theoretical developments in
control problems for MV-SDEs and recent applications of MV-SDEs in financial
mathematics. In this respect we mention the recent work of
\cite{pham2017dynamic} (see also the references therein), where a general form
of the Bellman principle is derived for the optimal control problems  under
McKean-Vlasov dynamics.
\par
Because of  dependence of the transition kernels on the marginal distributions, nonlinear
Markov processes can not be sampled like standard diffusion processes. Instead the so-called
interacting particle method combined with time discretisation is used to approximate  them. However,
unlike the standard Monte Carlo, the simulated particles are not independent. The
key result in the theory of interacting particle systems is the so-called propagation of
chaos result showing that the particles fulfil a kind of law of large numbers. In
particular, one can prove strong convergence of the interacting particle system to the
solution of the original McKean-Vlasov equation. Here we propose a fast
approximation method for optimal stopping problems related to (generally
multidimensional) MV- SDEs in spirit of the celebrated RMC method by Longstaff and Schwartz,
which is based on the underlying particle system of essentially dependent particles,
rather than a Monte Carlo sample of independent trajectories as in
\cite{J_LS2001}. In this respect one can speak about the Particle-Regression-Monte-Carlo
(PRMC) method. The convergence of this method is proved via a perturbation
analysis of a related linear regression problem due to an i.i.d. sample of the
original MV-SDE, and a recursive error analysis of the backward induction algorithm (in the spirit
of \cite{J_BelKolSch} and \cite{Z} for the case of independent samples). From
a mathematical point of view, this analysis may be considered as the main
contribution of the present paper.
Summing up, we provide a generic simulation based numerical approach for
solving optimal stopping problems with respect to a reward that depends on a 
 process following  multidimensional MV-SDE. Although this
problem may be relevant in a financial context, we note that  financial
terminology used in this paper is merely chosen for illustrative purposes, and
that the application scope of the developed methods is not restricted to finance.
\par
The structure of the paper is as follows. The general setup for optimal
stopping in a Markovian environment is presented in Section~\ref{mainsetup}.
In this section we also give a concise recap of the Longstaff-Schwartz and
Tsitsiklis van Roy method developed in \cite{J_LS2001} and \cite{J_TV2001},
respectively. Section~\ref{secMV} introduces Mckean-Vlasov equations and their
connection with particle systems, while Section~\ref{ApprDP} introduces a
particle version of regression based backward dynamic programming in the
spirit of \cite{J_LS2001}. In Section~\ref{regpart} we present one of our main
results, Theorem~\ref{mainth}, that deals with the convergence of the
regression approach applied to (generally dependent) particles. The
convergence of the PRMC algorithm is studied in Section~\ref{conver}.
Before proceeding to a rather general perturbation analysis in
Section~\ref{pertan} and proving Theorem~\ref{mainth} and
Theorem~\ref{thm_main} in Section~\ref{proofs}, we present some numerical
experiments in Section~\ref{numsec}.

\section{Optimal stopping, backward dynamic program}

\label{mainsetup} Let us explain the issue of optimal stopping in the context
of American options as a popular illustration. An American option grants its
holder the right to select time at which she exercises the option, i.e., calls
a pre-specified reward or cash-flow. This is in contrast to a European option
that may be exercised only at a fixed date. A general class of American option
pricing problems, i.e., optimal stopping problems respectively, can be
formulated with respect to an underlying $\mathbb{R}^{d}$-valued Markov
process $X$ $:=$ $\{X_{t},\,0\leq t\leq T\}$ defined on a filtered probability
space $(\Omega,\mathcal{F},(\mathcal{F}_{t})_{0\leq t\leq T},\mathrm{P})$ .
The process $X$ is assumed to be adapted to a filtration $(\mathcal{F}%
_{t})_{0\leq t\leq T}$ in the sense that each $X_{t}$ is $\mathcal{F}_{t}$
measurable. Recall that each $\mathcal{F}_{t}$ is a $\sigma$ -algebra of
subsets of $\Omega$ such that $\mathcal{F}_{s}\subseteq\mathcal{F}_{t}$ for
$s\leq t.$ In general, $X$ may describe any underlying physical or economical
quantity of interest such as, for example, the outside temperature, some (not
necessarily tradable) market index or price. Henceforth we restrict our
attention to the case where only a finite number $\mathcal{J}$ of stopping
(exercise) opportunities $0<t_{1}<t_{2}<\ldots<t_{\mathcal{J}}=T$ are allowed
(Bermudan options in financial terms). In this respect it should be noted that
a continuous exercise (American) option can be approximated by such a Bermudan
option with arbitrary accuracy, and so this is not a huge restriction. We now
consider a pre-specified reward $g_{j}(Z_{j})$ in terms of the discrete time
Markov chain
\[
Z_{i}:=X_{t_{i}},\quad i=0,\ldots,\mathcal{J},
\]
with $Z_{0}:=X_{t_{0}}:=X_{0},$ for some given functions $g_{1},\ldots
,g_{\mathcal{J}}$ mapping $\mathbb{R}^{d}$ into $[0,\infty).$ Note that for
technical convenience and without loss of generality we exclude $t_{0}=0$ from
the set of exercise dates. In a financial context we may assume that the
reward $g_{j}(Z_{j})$ is expressed in units of some (tradable) pricing
num\'{e}raire that has initial value $1$ Euro, say. That is, if exercised at
time $t_{j},\,j=1,\ldots,\mathcal{J}$, the option pays cash equivalent to
$g_{j}(Z_{j})$ units of the num\'{e}raire. Let $\mathcal{T}_{j}$ denote for
$j=0,...,\mathcal{J}$ the set of stopping times taking values in
$\{j,j+1,\ldots,\mathcal{J}\}\setminus\{0\}.$ In the spirit of contingent
claim pricing in finance we then assume that, for $j=0,...,\mathcal{J},$ a
fair price $V_{j}^{\ast}(z)$ of the corresponding Bermudan option at time
$t_{j}$ in state $z,$ given that the option was not exercised prior to $t_{j}%
$, is its value under the optimal exercise policy,
\begin{equation}
V_{j}^{\ast}(z)=\sup_{\tau\in\mathcal{T}_{j}}\mathsf{E}[g_{\tau}(Z_{\tau
})|Z_{j}=z]=\mathsf{E}[g_{\tau_{j}^{\ast}}(Z_{\tau_{j}^{\ast}})|Z_{j}=z],\quad
z\in\mathbb{R}^{d}, \label{stop}%
\end{equation}
hence the solution to an optimal stopping problem. In (\ref{stop}) we
introduced an optimal stopping family (policy), which can be also expressed in
the form
\begin{align}
\tau_{j}^{\ast} :=\min\left\{  j\leq l\leq\mathcal{J}:g_{l}(Z_{l})\geq
C_{l}^{\ast}(Z_{l})\right\}  , \label{sf}%
\end{align}
where
\begin{align}
C_{j}^{\ast}(z) :=\mathsf{E}[V_{j+1}^{\ast}(Z_{j+1})|Z_{j}=z],\quad
j=1,\ldots,\mathcal{J}-1,\text{ \ \ and \ \ }C_{\mathcal{J}}^{\ast}(z)\equiv0,
\label{cf}%
\end{align}
are so-called continuation functions. The process $V_{j}^{\ast}(Z_{j})$ is
called the \textit{Snell envelope} of the reward process $(g_{k}(Z_{k})).$
Both the stopping family (\ref{sf}) and the set of continuation functions
(\ref{cf}) satisfy the well-known \textit{Dynamic Programming Principle}. In
particular, for (\ref{sf}) we have the backward recursion
\begin{align*}
\tau_{\mathcal{J}}^{\ast}  &  =\mathcal{J},\\
\tau_{j}^{\ast}  &  =j\,1_{\left\{  g_{j}(Z_{j})\geq C_{j}^{\ast}%
(Z_{j})\right\}  }+\tau_{j+1}^{\ast}1_{\left\{  g_{j}(Z_{j})<C_{j}^{\ast
}(Z_{j})\right\}  },\quad j=1,\ldots,\mathcal{J}-1,
\end{align*}
and for (\ref{cf}) it holds,
\begin{align}
C_{\mathcal{J}}^{\ast}(z)  &  \equiv0,\label{dyn}\\
C_{j}^{\ast}(z)  &  =\mathsf{E}[g_{\tau_{j+1}^{\ast}}(Z_{\tau_{j+1}^{\ast}%
})|Z_{j}=z]\nonumber\\
&  =\mathsf{E}[\max(g_{j+1}(Z_{j+1}),C_{j+1}^{\ast}(Z_{j+1}))|Z_{j}=z],\quad
j=1,\ldots,\mathcal{J}-1.\nonumber
\end{align}
A common feature of almost all existing fast approximation algorithms is that
they deliver estimates $C_{N,1}(z),\ldots,C_{N,\mathcal{J}-1}(z)$ for the
continuation functions $C_{1}^{\ast}(z),$ $...,$ $C_{\mathcal{J}-1}^{\ast
}(z).$ Here the index $N$ indicates that the above estimates are based on the
set of independent \textquotedblleft training\textquotedblright\ trajectories
$\bigl(Z_{0}^{(i)},\ldots,Z_{\mathcal{J}}^{(i)}\bigr),$ $i=1,\ldots,N,$ all
starting from one point $Z_{0}$ i.e., $Z_{0}$ $=$ $Z_{0}^{(1)}=\ldots
=Z_{0}^{(N)}.$ In the case of regression methods the estimates for the
continuation values are obtained via regression based Monte Carlo. For
example, at step $\mathcal{J}-j$ one may estimate the expectation
\begin{equation}
\mathsf{E}[\max(g_{j+1}(Z_{j+1}),C_{N,j+1}(Z_{j+1}))\bigr)Z_{j}=z]
\label{regr_aim}%
\end{equation}
via linear regression based on the set of paths
\begin{equation}
\bigl(Z_{j}^{(i)},\max\{g_{j+1}(Z_{j+1}^{(i)}),C_{N,j+1}(Z_{j+1}%
^{(i)})\}\bigr),\quad i=1,\ldots,N, \label{TVr}%
\end{equation}
where $C_{N,j+1}(z)$ is the estimate for $C_{j+1}^{\ast}(z)$ obtained in the
previous step, and $C_{N,\mathcal{J}}(z):=0.$ This approach is basically the
method by Tsitsiklis and van Roy \cite{J_TV2001}. Alternatively, in the
so-called Longstaff-Schwarz algorithm, one mixes the estimates of the
continuation values with the corresponding estimates of the stopping times.
More precisely, on each trajectory $i=1,...,N,$ approximate stopping times
$\tau_{N,j}^{(i)},$ $j=1,\ldots,\mathcal{J},$ are recursively constructed by
first initializing $\tau_{N,\mathcal{J}}^{(i)}=\mathcal{J},$ for all $i.$
Then, once $\tau_{N,j+1}^{(i)},$ $i=1,...,N,$ is constructed for
$j<\mathcal{J},$ one computes from the sample
\begin{equation}
\bigl(Z_{j}^{(i)},g_{\tau_{N,j+1}^{(i)}}(Z_{\tau_{N,j+1}^{(i)}}^{(i)}%
)\bigr),\quad i=1,\ldots,N, \label{LSr}%
\end{equation}
an estimate $C_{N,j}(z)$ of the continuation function $C_{j+1}^{\ast}(z)$ by
projection on the linear span of a set of basis functions. Subsequently, the
approximate stopping times corresponding to exercise date $j$ on trajectories
$i=1,...,N,$ are defined via,
\begin{equation}
\tau_{N,j}^{(i)}=j\,1_{\bigl\{g_{j}(Z_{j}^{(i)})\geq C_{N,j}(Z_{j}%
^{(i)})\bigr\}}+\tau^{(i)}_{N,j+1}1_{\bigl\{g_{j}(Z_{j}^{(i)})<C_{N,j}%
(Z_{j}^{(i)})\bigr\}}. \label{tau}%
\end{equation}
Working all the way back, we thus obtain again a set of approximate
continuation functions $C_{N,1}(z),\ldots,C_{N,\mathcal{J}-1}(z)$ (note that
$C_{N,\mathcal{J}}(z)\equiv0$).

\begin{rem}
It should be noted that the algorithm based on (\ref{LSr}) is more popular
than the one based on (\ref{TVr}), because it behaves more stable in practice,
particularly when the number of exercise dates is getting very large.
Intuitively, this can be explained by the fact that the regression (\ref{LSr})
is always carried out on actually realized cash-flows, rather than on
approximative value functions as in (\ref{TVr}) which may become quite
unprecise because of error propagation due to a huge number of exercise dates,
see for example \cite{B1} for more rigorous arguments.
\end{rem}

Given the estimates $C_{N,1},\ldots,C_{N,\mathcal{J}-1}$, constructed by one
of the fast approximation methods above, we can construct a lower bound (low
biased estimate) for $V_{0}^{\ast}$ using the (generally suboptimal) stopping
rule:
\[
\tau_{N}=\min\bigl\{1\leq j\leq\mathcal{J}:g_{j}(Z_{j})\geq C_{N,j}%
(Z_{j})\bigr\}
\]
with $C_{N,\mathcal{J}}\equiv0$ by definition. Indeed, fix a natural number
$N_{\text{test}}$ and simulate $N_{\text{test}}$ new independent trajectories
of the process $Z.$ A low-biased estimate for $V_{0}^{\ast}$ can then be
defined as
\[
V_{0}^{N_{\text{test}},N}=\frac{1}{N_{\text{test}}}\sum_{r=1}^{N_{\text{test}%
}}g_{\tau_{N}^{(r)}}\bigl(Z_{\tau_{k}^{(r)}}^{(r)}\bigr)
\]
with
\[
\tau_{N}^{(r)}=\inf\Bigl\{1\leq j\leq\mathcal{J}:g_{j}(Z_{j}^{(r)})\geq
C_{N,j}(Z_{j}^{(r)})\Bigr\},\quad r=1,\ldots,N_{\text{test}}.
\]

\section{McKean-Vlasov equations and particle systems}

\label{secMV}

Let $[0,T]$ be a finite time interval and $(\Omega,\mathcal{F},\mathrm{P})$ be
a complete probability space, where a standard $m$-dimensional Brownian motion
$W$ is defined. We consider a class of McKean-Vlasov stochastic differential
equations (MVSDE), i.e., stochastic differential equations whose drift and
diffusion coefficients may depend on the current distribution of the process,
of the form:
\begin{equation}
\left\{
\begin{array}
[c]{ll}%
X_{t} & =\xi+\int_{0}^{t}\int_{\mathbb{R}^{d}}a(X_{s},y)\mu_{s}(dy)ds+\int
_{0}^{t}\int_{\mathbb{R}^{d}}b(X_{s},y)\mu_{s}(dy)dW_{s}\\
\mu_{t} & =\mathrm{Law}(X_{t}),\quad t\geq0,\quad X_{0}\sim\mu_{0}%
\end{array}
\right.  \label{eq:sde}%
\end{equation}
where $\mu_{0}$ is a distribution in $\mathbb{R}^{d},$ $a:\,\mathbb{R}%
^{d}\times\mathbb{R}^{d}\rightarrow\mathbb{R}^{d}$ and $b:\,\mathbb{R}%
^{d}\times\mathbb{R}^{d}\rightarrow\mathbb{R}^{d\times m}.$ A popular way of
simulating the MVSDE\ (\ref{eq:sde}) is to sample the so-called particle
system $\mathbf{X}_{t}^{N}=\bigl(X_{t}^{1,N},\ldots,X_{t}^{N,N}\bigr)\in
\mathbb{R}^{d\times N}$ from $N\times d$-dimensional SDE
\begin{equation}
X_{t}^{i,N}=\xi^{i}+\frac{1}{N}\sum_{j=1}^{N}\int_{0}^{t}a(X_{s}^{i,N}%
,X_{s}^{j,N})\,ds+\frac{1}{N}\sum_{j=1}^{N}\int_{0}^{t}b(X_{s}^{i,N}%
,X_{s}^{j,N})\,dW_{s}^{i} \label{eq:par}%
\end{equation}
for $i=1,\ldots,N,$ where $\xi^{i},$ $i=1,\ldots,N,$ are i.i.d copies of a
r.v. $\xi,$ distributed according the law $\mu_{0},$ and $W^{i},$ $i=1,...,N,$
are independent copies of $W.$ Under suitable assumptions (see, for example
~\cite{antonelli2002rate}) one has that%
\begin{equation}
\left\Vert \sup_{0\leq r\leq T}\left\vert X_{r}^{\cdot,N}-X_{r}^{\cdot
}\right\vert \right\Vert _{p}\leq C_{p}N^{-1/2}. \label{Koh}%
\end{equation}
In practice, $N\times d$-dimensional SDE system (\ref{eq:par}) cannot be
solved analytically either and one has to approximate its solution by a
suitable numerical integration scheme such as the Euler method, leading to a
next approximation $\mathbf{X}_{t}^{N,\delta}=\bigl(X_{t}^{1,N,\delta}%
,\ldots,X_{t}^{N,N,\delta}\bigr)$ with $\delta$ being the size of the each
Euler time step. Following \cite{antonelli2002rate}, one then has%
\begin{equation}
\left\Vert \sup_{0\leq r\leq T}\left\vert X_{r}^{\cdot,N,\delta}-X_{r}%
^{\cdot,N}\right\vert \right\Vert _{p}\lesssim\sqrt{\delta}, \label{Euler}%
\end{equation}
where $\lesssim$ involves a constant that does not depend on $N$ and $\delta.$

\begin{rem}
\label{exact} In order to fix the main ideas and to avoid a notational blow
up, we assume in this paper that the system $\mathbf{X}_{t}^{N}$ (cf.
(\ref{eq:par})) is constructed exactly, hence we neglect the numerical
integration error (\ref{Euler}). On the other hand, due to (\ref{Euler}) it
will be clear how several results in this paper have to be adapted in the case
where (\ref{eq:par}) is approximated using the Euler scheme.
\end{rem}

\begin{rem}
{In fact, the solution to the MVSDE (\ref{eq:sde}) may be considered as a
usual non-autonomous, Markovian diffusion, since $\left\{  \mu_{s}:0\leq s\leq
T\right\}  $ is some deterministic flow of distributions, although not
explicitly known beforehand. Therefore, we may consider the stopping problem
(\ref{stop}) with respect to the solution (\ref{eq:sde}), while the standard
notions of the Snell envelope and the Dynamic Programming Principle still
apply. However, in contrast to the standard diffusion processes $X,$ where
independent trajectories of $X$ may be simulated straightforwardly by Monte
Carlo, simulating of independent copies of (\ref{eq:sde}) is not directly
possible. As a way out, we will work with the particle system (\ref{eq:par})
of \textit{dependent} particles, instead of an ensemble of independent
trajectories of (\ref{eq:sde}).}
\end{rem}

\section{Dynamic programming on particle systems}

\label{ApprDP} In this section we describe a particle version of the
Longstaff-Schwarz regression algorithm due to (\ref{LSr}) and (\ref{tau}).
First we run the particle system $\mathbf{X}_{t}^{N}=\bigl(X_{t}^{1,N}%
,\ldots,X_{t}^{N,N}\bigr)\in\mathbb{R}^{d\times N}$ as described above and
set
\begin{equation}
Z_{j}^{i,N}=X_{j\Delta}^{i,N},\quad j=0,\ldots,\mathcal{J},\quad i=1,\ldots,N,
\label{Zt}%
\end{equation}
with $\mathcal{J}=\lfloor T/\Delta\rfloor.$ It should be noted that, unlike
Monte Carlo, the trajectories (\ref{Zt}) are generally \textit{dependent}. We
now consider an approximative dynamic programming algorithm based on the
(generally dependent) paths $\mathbf{Z}^{N}=(Z_{j}^{i,N},\,i=1,\ldots,N,$
$j=0,\ldots,\mathcal{J}).$ In the spirit of the Longstaff-Schwarz algorithm we
compute sequentially for $j=\mathcal{J},\ldots,1,$ approximate continuation
functions $C_{N,j}$ and approximate stopping times $\tau_{N,j}^{(i)},$
$i=1,...,N.$ That is, we initialize $\tau_{N,\mathcal{J}}^{(i)}=\mathcal{J},$
$i=1,...,N,$ and $C_{N,\mathcal{J}}=0,$ and once $\tau_{N,j+1}^{(i)},$
$i=1,...,N,$ and $C_{N,j+1},$ $\,j<\mathcal{J},$ are constructed, $C_{N,j}$ is
obtained from solving the minimization,%
\begin{equation}
C_{N,j}:=\underset{h\in\mathcal{H}_{K}}{\arg\min}\left\{  \frac{1}{N}%
\sum_{i=1}^{N}\left(  g_{\tau_{N,j+1}^{(i)}}(Z_{\tau_{N,j+1}^{(i)}}%
^{i,N})-h(Z_{j}^{i,N})\right)  ^{2}\right\}  \label{Cd}%
\end{equation}
($C_{N,\mathcal{J}}:=0$). Next $\tau_{N,j}^{(i)}$ is updated analogue to the
scheme (\ref{tau}), i.e.,%
\[
\tau_{N,j}^{(i)}=j\,1_{\bigl\{g_{j}(Z_{j}^{i,N})\geq C_{N,j}(Z_{j}%
^{i,N})\bigr\}}+\tau_{N,j+1}1_{\bigl\{g_{j}(Z_{j}^{i,N})<C_{N,j}(Z_{j}%
^{i,N})\bigr\}}.
\]
Note that, for fixed $j,$ the pairs%
\begin{equation}
\left(  Z_{j}^{i,N},g_{\tau_{N,j+1}^{(i)}}(Z_{\tau_{N,j+1}^{(i)}}%
^{i,N})\right)  ,\text{ \ \ }i=1,...,N, \label{pas}%
\end{equation}
are generally dependent, but, have the same distribution for each $i.$ As such
(\ref{Cd}) is indeed can be viewed as an estimator of $C_{j}^{\ast}$ (cf.
(\ref{dyn})). Usually the space $\mathcal{H}_{K}$ is taken to be the linear
span of some given set of basis functions, so that the minimization (\ref{Cd})
boils down to a linear least squares problem that can be solved via
straightforward linear algebra. Let us refer to the above algorithm as the
PRMC algorithm.

\section{Regression  on interacting particle systems}

\label{regpart}

In this section we consider for some fixed $t\geq0$ and time $T>t$ a generic
problem of computing the functionals of the form
\begin{equation}
w(x)=\mathsf{E}\left[  f\left(  X_{T}\right)  \mid X_{t}=x\right]  ,
\label{ef}%
\end{equation}
where $(X_{t})$ is the solution of \eqref{eq:sde}. In (\ref{ef}), $T$ may in
general be any random time. Let $\mathbf{X}_{t}^{N}=(X_{t}^{1,N},\ldots
,X_{t}^{N,N})$ be a particle system (\ref{eq:par}). Furthermore, let for each
$K\in\mathbb{N},$ $\mathcal{H}_{K}$ be a $K$-dimensional linear space of
functions $h:\mathbb{R}^{d}\rightarrow\mathbb{R}$ and consider the estimate
\begin{equation}
\widetilde{w}_{N}:=\underset{h\in\mathcal{H}_{K}}{\arg\min}\left\{  \frac
{1}{N}\sum_{i=1}^{N}\left(  f(X_{T}^{i,N})-h(X_{t}^{i,N})\right)
^{2}\right\}  ,
\end{equation}
where the dimension $K$ may depend on $N.$ Let $\left(  \psi_{k}\right)
_{k=1,2,...}$ be a sequence of linearly independent basis functions and set
$\mathcal{H}_{K}:=\mathsf{span}\left\{  \psi_{1},\ldots,\psi_{K}\right\}  .$
In this section we are going to analyze the properties of the estimate
$\widetilde{w}_{N}.$ Note that the random variables $X_{T}^{1,N},\ldots
,X_{T}^{N,N}$ are generally dependent, so that the known results from
regression analysis (see, e.g. \cite{gyorfi2002distribution}) can not be
applied directly. Consider the truncated version of the estimate
$\widetilde{w}_{N}$ defined as $T_{M}\widetilde{w}_{N},$ where $T_{M}$ is a
truncation operator defined for a generic function $h$ and a threshold $M$ as
\begin{equation}
T_{M}h=%
\begin{cases}
M, & h>M,\\
h, & -M\leq f\leq M,\\
-M, & h<-M.
\end{cases}
\label{trunc}%
\end{equation}
We have the following theorem.

\begin{thm}
\label{mainth} Assume that all functions $\psi_{k},$ $k=1,2,\ldots$ and $f$
are globally bounded and Lipschitz continuous. That is, there exist constants
$M_{f},$ $L_{f},$ $L_{k},$ $M_{k},\ell_{k},$ $k=1,2,\ldots.$ such that for all
$x,y\in\mathbb{R}^{d},$
\begin{align}
|f(x)|  &  \leq M_{f},\quad\frac{1}{K}\sum_{k=1}^{K}|\psi_{k}(x)|^{2}\leq M_{K}%
^{2}, \label{eq:bound_psi}\\
\quad\left\vert \psi_{k}(x)-\psi_{k}(y)\right\vert  &  \leq L_{k}\left\vert
x-y\right\vert ,\quad\left\vert f(x)-f(y)\right\vert \leq L_{f}\left\vert
x-y\right\vert. \label{eq:psi_lip}
\end{align}
Further suppose that
\begin{equation}
0<\varkappa_{\circ}\leq\lambda_{\min}\left(  \Sigma_{K}\right)  <\lambda
_{\max}\left(  \Sigma_{K}\right)  \leq\varkappa^{\circ}<\infty
\label{eq: lambda_ass}%
\end{equation}
for all $K\in\mathbb{N},$ where
\[
\Sigma_{K}=\left(  \int\psi_{k}\left(  x\right)  \psi_{l}\left(  x\right)
\mu_{t}(dx),\,k,l=1,\ldots,K\right)  .
\]
Then, it holds
\begin{multline}
\Vert T_{M_{f}}\widetilde{w}^{N}(x)-w\left(  x\right)  \Vert_{L_{2}(\mu_{t}%
)}\\
\lesssim\frac{M_{K}\sqrt{K}}{\sqrt{N}}\Bigl[d_{1}M_{f}\ell_{K}+d_{2}%
L_{f}\Bigr]+\frac{M_{f}}{\sqrt{N}}\Bigl[d_{3}\ell_{K}+\sqrt{1+\log N}\sqrt
{K}\Bigr]\\
+M_{f}\sqrt{K}\exp\left[  -d_{4}\frac{N}{KM_{K}^{2}}\right]  +\inf
_{h\in\mathcal{H}_{K}}\Vert h-w\Vert_{L_{2}(\mu_{t})}\label{res}%
\end{multline}
with 
\begin{eqnarray}
\ell_{K}^{2}    :=\sum_{k=1}^{K}L_{k}^{2}. \label{lk}
\end{eqnarray}
In \eqref{res} the constants $d_{1,2,3,4}$ depend on $\varkappa_{\circ},$
$\varkappa^{\circ}$ only, $\lesssim$ denotes $\leq$ up to a universal constant
for each term.
\end{thm}

Theorem~\ref{mainth} will be proved in Section~\ref{proofs}.

\section{Convergence analysis of the PRMC algorithm}

\label{conver}

In this section we investigate the convergence properties of the PRMC
regression algorithm. To this end, we modify the PRMC algorithm in such a way
that our fundamental result, Theorem~\ref{mainth}, may be applied. In fact, we
follow an approach in the spirit of \cite{Z} (cf. \cite{J_BelKolSch}), and
assume that instead of one particle sample $\mathbf{Z}^{N},$ we have at hand
for $j=1,\ldots,\mathcal{J}-1,$ independent particle samples $\mathbf{Z}%
^{j;N}:=(Z_{r}^{j;i,N},\,i=1,\ldots,N,$ $r=0,\ldots,\mathcal{J}),$ all
starting at $Z_{0}=X_{0}.$ We next assume that $C_{N,j}$ and $\tau_{N,j}$ are
constructed in the following backward recursive way: As initialization we set
$\tau_{N,\mathcal{J}}=\mathcal{J}$ and $C_{N,\mathcal{J}}\equiv0.$ Once
$C_{N,j+1}$ and $\tau_{N,j+1},$ $j+1\leq\mathcal{J},$ are determined based on
the samples $\mathbf{Z}^{j+1;N},\ldots,\mathbf{Z}^{\mathcal{J}-1;N},$ then
$C_{N,j}$ is constructed based on the samples $\mathbf{Z}^{j;N},\ldots
,\mathbf{Z}^{\mathcal{J}-1;N},$ via%
\[
C_{N,j}=\underset{h\in\mathcal{H}_{K}}{\arg\min}\sum_{i=1}^{N}\left(
g_{\tau^{(i)}_{N,j+1}}\bigl(  Z_{\tau^{(i)}_{N,j+1}}^{j;i,N}\bigr)
-h\bigl(  Z_{j}^{j;i,N}\bigr)  \right)  ^{2},
\]
and, subsequently, $\tau_{N,j}$ is defined by
\[
\tau_{N,j}:=j\,1_{\bigl\{g_{j}(Z_{j})\geq C_{N,j}(Z_{j})\bigr\}}+\tau
_{N,j+1}1_{\bigl\{g_{j}(Z_{j})<C_{N,j}(Z_{j})\bigr\}},
\]
for a generic dummy trajectory $\left(  Z_{l}\right)  _{l=0,\ldots
,\mathcal{J}}$ corresponding to the (exact) solution of (\ref{eq:sde})
independent of
\[
\mathcal{G}_{j}:=\sigma\left\{  \mathbf{Z}^{j;N},\ldots,\mathbf{Z}%
^{\mathcal{J}-1;N}\right\}  .
\]
Let us further define
\begin{equation}
\overline{C}_{N,j}(z):=\mathsf{E}_{\mathcal{G}_{j+1}}\left[  \left.
g_{\tau_{N,j+1}}\left(  Z_{\tau_{N,j+1}}\right)  \right\vert Z_{j}=z\right]  .
\label{cbar}%
\end{equation}
Note that the random function $\overline{C}_{N,j}$ is $\mathcal{G}_{j+1}%
$-measurable while its estimate $C_{N,j}$ is $\mathcal{G}_{j}$-measurable. By
running this procedure all the way down to $j=1,$ we so end up with a sequence
of approximative continuation functions $C_{N,j}\left(  \cdot\right)  ,$ and
the corresponding conditional expectations $\overline{C}_{N,j}(z).$ The
following lemma holds.

\begin{lem}
\label{lem23} For the conditional expectations (\ref{cbar}) we have that,%
\begin{equation}
\left\Vert \overline{C}_{N,j} -C_{j}^{\ast} \right\Vert _{L_{p}(\mu_{j})}%
\leq\sum_{l=j+1}^{\mathcal{J}-1}\left\Vert C_{N,l} -C_{l}^{\ast} \right\Vert
_{L_{p}(\mu_{l})} \label{eq:bound_1}%
\end{equation}
with $p\geq1$ by slightly abusing notation and using $\mu_{j}=\mu_{t_{j}}.$
Note that the inequality (\ref{eq:bound_1}) involves $\mathcal{G}_{j+1}%
$-measurable objects.
\end{lem}

\begin{rem}
It is interesting to compare the estimate \eqref{eq:bound_1} with similar ones
in Lemma 2.3 of \cite{Z}.
\end{rem}

The following theorem states the convergence of the approximate continuation
functions in the PRMC algorithm to the exact ones, respectively.

\begin{thm}
\label{thm_main} Assume that the conditions  
\eqref{eq:bound_psi}, \eqref{eq:psi_lip}, \eqref{eq: lambda_ass} are fulfilled with $f$ replaced
by $g_{j}$ uniformly in $j=1,\ldots,\mathcal{J}$. By denoting the norm%
\[
\left\Vert \cdot\right\Vert _{L_{2}(\mu_{j},\mathbb{P})}^{2}:=\mathsf{E}%
\left[  \left\Vert \cdot\right\Vert _{L_{2}(\mu_{j})}^{2}\right]  ,
\]
due to the unconditional expectation with respect to the \textquotedblleft all
in\textquotedblright\ probability measure $\mathbb{P},$ one has for
$j=1,\ldots,\mathcal{J}-1,$%
\begin{equation}
\left\Vert C_{N,\mathcal{J}-j}-C_{\mathcal{J}-j}^{\ast}\right\Vert _{L_{2}%
(\mu_{\mathcal{J}-j},\mathbb{P})}\leq\Delta_{N,K}(2+c)^{j},\label{eq: main}%
\end{equation}
where $\Delta_{N,K}=c\bigl(\epsilon_{N,K}+\max_{j}\inf_{h\in\mathcal{H}_{K}%
}\left\Vert C_{j}^{\ast}-h\right\Vert _{L_{2}(\mu_{j})}\bigr)$ for some $c>0$
with
\[
\epsilon_{N,K}=\sqrt{\frac{K}{N}}\left(  M_{K}\ell_{K}+\eta_{1}\sqrt{1+\log
N}\right)  +\eta_{2}\sqrt{K}\exp\left[  -\eta_{3}\frac{N}{KM_{K}^{2}}\right]
\]
for some constants $\eta_{1,2,3}>0$ not depending on $K$ and $N.$
\end{thm}

\begin{example}
The Hermite polynomial of order $j$ is given, for $j\geq0$, by:
\[
H_{j}(x)=(-1)^{j}e^{x^{2}}\frac{d^{j}}{dx^{j}}(e^{-x^{2}}).
\]
Hermite polynomials are orthogonal with respect to the weight function
$e^{-x^{2}}$ and satisfy: $\int_{{\mathbb{R}}}H_{j}(x)H_{\ell}(x)e^{-x^{2}%
}dx=2^{j}j!\sqrt{\pi}\delta_{j,\ell}.$ The Hermite function of order $j$ is
given by:
\begin{equation}
\psi_{j}(x)=c_{j}H_{j}(x)e^{-x^{2}/2},\quad c_{j}=\left(  2^{j}j!\sqrt{\pi
}\right)  ^{-1/2}.\label{fhermite}%
\end{equation}
The sequence $(\psi_{j},j\geq0)$ is an orthonormal basis of ${L}%
_{2}({\mathbb{R}})$. The infinite norm of $\psi_{j}$ satisfies (see
Szeg\"{o}~\cite{szego1975orthogonal} p.242):
\begin{equation}
\,\Vert\psi_{j}\Vert_{\infty}\leq M_{0},\quad\quad M_{0}\simeq1,086435/\pi
^{1/4}\simeq0.8160.\label{borneh}%
\end{equation}
Furthermore, since $\psi_{k}^{\prime}(x)=\sqrt{k/2}\psi_{k-1}(x)-\sqrt
{(k+1)/2}\psi_{k+1}(x)$ we derive that the condition \eqref{eq:psi_lip} is
fulfilled with $L_{k}=2M_{0}\sqrt{k}.$
\end{example}

\section{Numerical experiment}

\label{numsec} As a simple illustration of the proposed methodology, let us
consider optimal stopping problem in the so-called Shimizu-Yamada model%
\begin{equation}
dX_{t}=\left(  a\mathsf{E}\left[  X_{t}\right]  +bX_{t}\right)  \, dt+\sigma
dW_{t},\quad X_{0}=x_{0}, \quad t\in[0,T] \label{SY}%
\end{equation}
(see \cite{frank2004stochastic}, Section 3.10), which has the explicit
solution%
\begin{equation}
X_{t}=x_{0}e^{\left(  a+b\right)  t}+\sigma\int_{0}^{t}e^{b(t-s)}dW_{s}
\label{Gs}%
\end{equation}
that in turn solves the ordinary SDE%
\begin{equation}
dX_{t}=(x_{0}ae^{\ \left(  a+b\right)  t}+bX_{t})\,dt+\sigma dW_{t}
\label{oSDE}%
\end{equation}
(cf. \cite{belomestny2017projected}). It is straightforward to show that the
conditional mean and variance of (\ref{Gs}) are given by
\begin{align*}
\mathsf{E}\left[  X_{t}|X_{s}\right]   &  =e^{b\left(  t-s\right)  }%
X_{s}+x_{0}e^{bt}\left(  e^{at}-e^{as}\right)  \text{ \ \ and}\\
\mathrm{Var}\left[  X_{t}|X_{s}\right]   &  =\sigma^{2}\frac{e^{2b\left(
t-s\right)  }-1}{2b},\text{ \ \ }0\leq s\leq t,
\end{align*}
respectively. That is, for $s=0$ we have that%
\begin{equation}
\mathsf{E}\left[  X_{t}\right]  =x_{0}e^{\left(  a+b\right)  t}\text{ \ \ and
\ \ }\mathrm{Var}\left[  X_{t}\right]  =\sigma^{2}\frac{e^{2bt}-1}{2b}.
\label{EV}%
\end{equation}
In the particular case $b=-a,$ $a>0$ one so has
\begin{align}
\mathsf{E}\left[  X_{t}|X_{s}\right]   &  =e^{-a\left(  t-s\right)  }%
X_{s}+x_{0}\left(  1-e^{-a\left(  t-s\right)  }\right)  ,\label{CondD}\\
\mathrm{Var}\left[  X_{t}|X_{s}\right]   &  =\sigma^{2}\frac{1-e^{-2a\left(
t-s\right)  }}{2a},\text{ \ \ }0\leq s\leq t,\text{ \ \ }a>0,\nonumber
\end{align}
and (\ref{EV}) yields $\mathsf{E}\left[  X_{t}\right]  =x_{0}$ for all $t.$
For this case the particle system (\ref{eq:par}) reads%
\begin{equation}
X_{t}^{i,N}=x_{0}+\frac{a}{N}\sum_{j=1}^{N}\int_{0}^{t}X_{s}^{j,N}ds-a\int
_{0}^{t}X_{s}^{i,N}ds+\sigma W_{t}^{i},\quad t\in[0,T]. \label{POU}%
\end{equation}
We now consider the optimal stopping problem%
\begin{equation}
V_{0}^{\ast}=\sup_{\tau\in\mathcal{T}_{j}}\mathsf{E}[g_{\tau}(X_{t_{\tau}})],
\label{stopn}%
\end{equation}
for some reward functions $g_{j}:\mathbb{R}\rightarrow\mathbb{R}_{\geq0},$
$t_{j}=jT/\mathcal{J},$ where $(X_{t})$ solves (\ref{SY}) with $b=-a,$ $a>0.$
Since $X$ follows an ordinary SDE (\ref{oSDE}), we may compute an
approximation to (\ref{stopn}) numerically by the Longstaff-Schwarz method
\cite{J_LS2001} based on independent trajectories of (\ref{oSDE}), and then
compare it to the particle based Longstaff-Schwarz algorithm proposed in
Section~\ref{ApprDP}. Let us further consider, for illustration a Bermudan put
option (in financial terms),%
\[
g_{j}(x)=e^{-rt_{j}}(x-K)^{+},\quad j=0,\ldots,\mathcal{J},
\]
for some $K>0,$ where $r$ can be interpreted as interest rate. We take the
following parameters $d=1,$ $x_{0}=1,$$K=0.1,$ $a=1,$ $T=1,$ $\mathcal{J}=100$
and implement the following two phase algorithm. In the first stage we run
$N_{\mathrm{tr}}$ trajectories either of the particle system \eqref{POU} or of
the process \eqref{oSDE}. Using these trajectories, we estimate the
corresponding continuation functions using linear regression with quadratic
polynomials and reward functions as basis functions. In the second stage we
use the estimated continuation values on a new set of $N_{\mathrm{test}}=5000$
testing trajectories to construct a suboptimal stopping rule and consequently
a lower bound for $V_{0}^{*}$ by averaging over the testing paths. We also
compute dual upper bounds using the estimated continuation functions and
$N_{\mathrm{in}}=100$ inner paths to approximate one step conditional
expectations, see, e.g. Chapter 3 in \cite{belomestny2018book}. The results
for different values of $N_{\mathrm{tr}}$ are shown in Table~\ref{TablePRMC}.
\begin{table}[h]
\centering
\renewcommand{\arraystretch}{1.3}
\begin{tabular}
[c]{|c|c|c|}\hline
$N_{\mathrm{tr}}$ & RMC & PRMC\\\hline\hline
10 & [0.9393(0.0079), 1.2742(0.0076)] & [0.9287(0.0058),
1.1750(0.0038)]\\\hline
50 & [1.0047(0.0082), 1.0942(0.0019)] & [0.9829(0.0072),
1.1745(0.0041)]\\\hline
100 & [1.0144(0.0073), 1.0871(0.0013)] & [1.0079(0.0080),
1.0978(0.0023)]\\\hline
300 & [1.0342(0.0077), 1.0718(0.0009)] & [1.0330(0.0070),
1.0700(0.0010)]\\\hline
1000 & [1.0575(0.0075), 1.0699(0.0007)] & [1.0546(0.0078),
1.0689(0.0008)]\\\hline
\end{tabular}
\caption{Lower and dual upper bounds with standard deviations for RMC and PRMC
algorithms.}%
\label{TablePRMC}%
\end{table}As can be seen from Table~\ref{TablePRMC}, the PRMC (Particle
Regression MC) performs a bit worse than RMC (Regression MC), but the
difference becomes smaller as $N_{\mathrm{tr}}$ increases.

\section{Perturbation analysis for linear regression}

\label{pertan}

Consider a least squares problem of the form
\begin{equation}
\beta^{\circ}=\underset{\beta\in\mathbb{R}^{d}}{\arg\min}\sum_{i=1}^{N}%
(Y_{i}-\beta^{\top}U_{i})^{2}, \label{eq:least_squares}%
\end{equation}
where for $i=1,...,N,$ $\left(  Y_{i},U_{i}\right)  $ are i.i.d. pairs of a
random variable $Y_{i}$ and a random (column) vector $U_{i}\in$ $\mathbb{R}%
^{d}.$ With $U:=(U_{1},\ldots,U_{N})\in\mathbb{R}^{d\times N},$ $Z=N^{-1/2}%
U^{\top},$ and $V=N^{-1/2}\left(  Y_{1},\ldots,Y_{N}\right)  ^{\top},$ the
solution of the problem (\ref{eq:least_squares}) can be written in terms of
pseudo inverses (denoted with $\dag$),%
\begin{equation}
\beta^{\circ}=\left(  UU^{\top}\right)  ^{-1}UY=\left(  Z^{\top}Z\right)
^{-1}Z^{\top}V=Z^{\dag}V. \label{bet}%
\end{equation}
Consider now the least squares problem (\ref{eq:least_squares}) due to a
perturbation $\left(  \widetilde{Y}_{i},\widetilde{U}_{i}\right)  $ of the
pairs $\left(  Y_{i},U_{i}\right)  ,$ and define $\widetilde{Z}$ and
$\widetilde{V}$ accordingly. We so consider (cf. (\ref{bet}))%
\begin{equation}
\widetilde{\beta}^{\circ}=\left(  \widetilde{Z}^{\top}\widetilde{Z}\right)
^{-1}\widetilde{Z}^{\top}\widetilde{V}=\widetilde{Z}^{\dag}\widetilde{V}
\label{bett}%
\end{equation}
and set%
\begin{equation}
\widetilde{Z}=Z+E,\quad\widetilde{V}=V+F. \label{pert}%
\end{equation}
While the rows of $Z$ and the components of $V$ are independent, the rows of
the perturbation matrix $E$ and the components of the perturbation vector $F$
are generally dependent. \ Also we note that we don't assume any kind of
independence between the perturbations $E$ and $F$ and the matrix $Z$ and
vector $V,$ respectively.

\begin{thm}
\label{thm:perturbation} Consider the least squares problem
(\ref{eq:least_squares}) with solution (\ref{bet}), and its perturbation due
to (\ref{pert}) with solution (\ref{bett}), respectively. Assume that
$U_{1},\ldots,U_{N}$ in (\ref{eq:least_squares}) are i.i.d. random vectors in
$\mathbb{R}^{d}$ such that for some $M>0,$ $\left\Vert U_{1}\right\Vert \leq
M$ a.s. Set
\[
\mathsf{E}\left[  U_{1}U_{1}^{\top}\right]  =\Sigma,
\]
so that
\[
Z^{\top}Z=\frac{1}{N}UU^{\top}=\frac{1}{N}\sum_{i=1}^{N}U_{i}U_{i}^{\top}.
\]
Let $\lambda_{\min}(\Sigma)$ be the smallest eigenvalue, and $\lambda_{\max
}(\Sigma)$ be the largest eigenvalue of $\Sigma,$ respectively. Then for any
$\rho\in(0,\lambda_{\min}(\Sigma))$ and $\varepsilon\in(0,\lambda_{\min
}\left(  \Sigma\right)  -\rho)$ we have on the set $\mathcal{C}:=\mathcal{C}%
_{1}\cap\mathcal{C}_{2}\cap\mathcal{C}_{3}\cap\mathcal{C}_{4}$ with%
\begin{align*}
\mathcal{C}_{1}  &  :\text{ \ \ }\lambda_{\mathrm{\max}}\left(  Z^{\top
}Z\right)  <\lambda_{\mathrm{\max}}\left(  \Sigma\right)  +\varepsilon,\\
\mathcal{C}_{2}  &  :\text{ \ \ }\lambda_{\min}\left(  Z^{\top}Z\right)
>\lambda_{\mathrm{\min}}\left(  \Sigma\right)  -\varepsilon,\\
\mathcal{C}_{3}  &  :\text{ \ \ }\lambda_{\mathrm{\min}}\left(  \Sigma\right)
-\left(  2\sqrt{\lambda_{\mathrm{\max}}\left(  \Sigma\right)  +\varepsilon
}+1\right)  \Vert E\Vert>\rho+\varepsilon,\\
\mathcal{C}_{4}  &  :\text{ \ \ }\Vert E\Vert<1.
\end{align*}
that%
\[
\Vert\widetilde{\beta}^{\circ}-\beta^{\circ}\Vert\leq c_{1}(\Sigma
,\varepsilon,\rho)\Vert E\Vert\Vert V\Vert+c_{2}(\Sigma,\varepsilon,\rho)\Vert
F\Vert,
\]
where%
\begin{align*}
c_{1}(\Sigma,\varepsilon,\rho)  &  :=\frac{1}{\rho}+\frac{2\left(
\lambda_{\mathrm{\max}}\left(  \Sigma\right)  +\varepsilon\right)
+\sqrt{\lambda_{\mathrm{\max}}\left(  \Sigma\right)  +\varepsilon}}{\rho^{2}%
}\text{ \ \ and }\\
c_{2}(\Sigma,\varepsilon,\rho)  &  :=c_{1}(\Sigma,\varepsilon,\rho
)+\frac{\sqrt{\lambda_{\mathrm{\max}}\left(  \Sigma\right)  +\varepsilon}%
}{\lambda_{\mathrm{\min}}\left(  \Sigma\right)  -\varepsilon}.
\end{align*}
Furthermore, for any $\delta\in(0,1),$ and $N$ such that%
\begin{equation}
\varepsilon=\varepsilon_{\delta,N}=M\sqrt{\,\frac{\log(2d/\delta)}{NC}}%
\frac{\lambda_{\max}^{3/2}(\Sigma)}{\lambda_{\min}(\Sigma)}\leq\lambda
_{\mathrm{\min}}(\Sigma)-\rho\label{eN}%
\end{equation}
(cf. (\ref{epdel})), one has for the probability of $\mathcal{C},$
\[
\mathsf{P}\left[  \mathcal{C}\right]  \geq1-\delta-C_{p}\left(  \left(
\frac{2\sqrt{\lambda_{\mathrm{\max}}\left(  \Sigma\right)  +\varepsilon}%
+1}{\lambda_{\mathrm{\min}}\left(  \Sigma\right)  -\varepsilon-\rho}\right)
^{p}+1\right)  ,
\]
provided $\delta$ and $C_{p}:=\mathsf{E}\left[  \left\vert E\right\vert
^{p}\right]  $ are small enough (such that the above bound is positive).
\end{thm}

\begin{proof}
Note that $\mathcal{C}:=\mathcal{C}_{1}\cap\mathcal{C}_{2}\cap\mathcal{C}%
_{3}\cap\mathcal{C}_{4}$ implies (\ref{cond}) in Lemma \ref{App} and so by
this Lemma,%
\begin{align*}
\Vert(Z+E)^{\dagger}-Z^{\dagger}\Vert &  \leq\frac{\Vert E\Vert}{\rho}\left[
1+\frac{\left(  2\Vert Z\Vert+1\right)  \Vert Z\Vert}{\rho}\right] \\
&  \leq\frac{\Vert E\Vert}{\rho}\left[  1+\frac{2\left(  \lambda
_{\mathrm{\max}}\left(  \Sigma\right)  +\varepsilon\right)  +\sqrt
{\lambda_{\mathrm{\max}}\left(  \Sigma\right)  +\varepsilon}}{\rho}\right] \\
&  =c_{1}(\Sigma,\varepsilon,\rho)\Vert E\Vert
\end{align*}
Thus, on $\mathcal{C}$ one has also,%
\begin{align*}
\Vert(Z+E)^{\dagger}\Vert &  \leq\Vert(Z+E)^{\dagger}-Z^{\dagger}\Vert+\Vert
Z^{\dagger}\Vert\\
&  \leq c_{1}(\Sigma,\varepsilon,\rho)+\frac{\sqrt{\lambda_{\mathrm{\max}%
}\left(  \Sigma\right)  +\varepsilon}}{\lambda_{\mathrm{\min}}\left(
\Sigma\right)  -\varepsilon}=c_{2}(\Sigma,\varepsilon,\rho),
\end{align*}
using that $\Vert Z^{\dagger}\Vert\leq\Vert(Z^{\top}Z)^{-1}\Vert\Vert Z\Vert.$
So on $\mathcal{C}$ we get,%
\begin{align*}
\Vert\widetilde{\beta}^{\circ}-\beta^{\circ}\Vert &  =\Vert\widetilde
{Z}^{\dagger}\widetilde{V}-Z^{\dagger}V\Vert=\Vert(Z+E)^{\dagger
}(V+F)-Z^{\dagger}V\Vert\\
&  \leq c_{1}(\Sigma,\varepsilon,\rho)\Vert E\Vert\Vert V\Vert+c_{2}%
(\Sigma,\varepsilon,\rho)\Vert F\Vert.
\end{align*}
For the probability of $\mathcal{C}$ one has%
\begin{equation}
\mathsf{P}\left[  \mathcal{C}\right]  \geq1-\mathsf{P}\left[  \Omega
\backslash\mathcal{C}_{1}\cup\Omega\backslash\mathcal{C}_{2}\right]
-\mathsf{P}\left[  \Omega\backslash\mathcal{C}_{3}\right]  -\mathsf{P}\left[
\Omega\backslash\mathcal{C}_{4}\right]  . \label{esp}%
\end{equation}
For the term $\mathsf{P}\left[  \Omega\backslash\mathcal{C}_{1}\cup
\Omega\backslash\mathcal{C}_{2}\right]  $ we are going to apply Lemma
\ref{lem:spectrum_subgaussian}. It is easy to see that, since $0<\lambda
_{\mathrm{\min}}(\Sigma)\leq\lambda_{\mathrm{\max}}(\Sigma),$ (\ref{eN})
implies (\ref{epdelc}) in Lemma~\ref{lem:spectrum_subgaussian}. So, due to
this lemma, we have that
\[
\mathsf{P}\left[  \Omega\backslash\mathcal{C}_{1}\cup\Omega\backslash
\mathcal{C}_{2}\right]  \leq\delta.
\]
Furthermore,%
\begin{align*}
\mathsf{P}\left[  \Omega\backslash\mathcal{C}_{3}\right]   &  =\mathsf{P}%
\left[  \lambda_{\mathrm{\min}}\left(  \Sigma\right)  -\left(  2\sqrt
{\lambda_{\mathrm{\max}}\left(  \Sigma\right)  +\varepsilon}+1\right)  \Vert
E\Vert\leq\rho+\varepsilon\right] \\
&  =\mathsf{P}\left[  \frac{\lambda_{\mathrm{\min}}\left(  \Sigma\right)
-\varepsilon-\rho}{2\sqrt{\lambda_{\mathrm{\max}}\left(  \Sigma\right)
+\varepsilon}+1}\leq\Vert E\Vert\right] \\
&  \leq\left(  \frac{2\sqrt{\lambda_{\mathrm{\max}}\left(  \Sigma\right)
+\varepsilon}+1}{\lambda_{\mathrm{\min}}\left(  \Sigma\right)  -\varepsilon
-\rho}\right)  ^{p}\mathsf{E}\left[  \Vert E\Vert^{p}\right]  ,
\end{align*}
and%
\[
\mathsf{P}\left[  \Omega\backslash\mathcal{C}_{4}\right]  =\mathsf{P}\left[
\Vert E\Vert\geq1\right]  \leq\mathsf{E}\left[  \Vert E\Vert^{p}\right]  .
\]
The statement now follows from (\ref{esp}).
\end{proof}

\begin{cor}
\label{cor9} Let us take $\varepsilon=\rho=\lambda_{\mathrm{\min}}\left(
\Sigma\right)  /4.$ Then with%
\begin{align*}
c_{1}(\Sigma)  &  :=\frac{1}{\lambda_{\min}\left(  \Sigma\right)  /4}\\
&  +\frac{2\left(  \lambda_{\mathrm{\max}}\left(  \Sigma\right)
+\lambda_{\min}\left(  \Sigma\right)  /4\right)  +\sqrt{\lambda_{\mathrm{\max
}}\left(  \Sigma\right)  +\lambda_{\min}\left(  \Sigma\right)  /4}}{\left(
\lambda_{\min}\left(  \Sigma\right)  /4\right)  ^{2}},\\
c_{2}(\Sigma)  &  :=c_{1}(\Sigma)+\frac{\sqrt{\lambda_{\mathrm{\max}}\left(
\Sigma\right)  +\lambda_{\min}\left(  \Sigma\right)  /4}}{\lambda
_{\mathrm{\min}}\left(  \Sigma\right)  -\lambda_{\min}\left(  \Sigma\right)
/4},
\end{align*}
we have on $\mathcal{C},$%
\[
\left\Vert \widetilde{\beta}^{\circ}-\beta^{\circ}\right\Vert \leq
c_{1}(\Sigma)\left\Vert E\right\Vert \left\Vert V\right\Vert +c_{2}%
(\Sigma)\left\Vert F\right\Vert ,
\]
with probability%
\begin{align*}
\mathsf{P}\left[  \mathcal{C}\right]   &  \geq1-2d\exp\left[  -N\frac
{C\lambda_{\min}^{4}(\Sigma)}{16M^{2}\lambda_{\max}^{3}(\Sigma)}\right] \\
&  -C_{p}\left(  \left(  \frac{2\sqrt{\lambda_{\mathrm{\max}}\left(
\Sigma\right)  +\lambda_{\mathrm{\min}}\left(  \Sigma\right)  /4}+1}%
{\lambda_{\mathrm{\min}}\left(  \Sigma\right)  /2}\right)  ^{p}+1\right)  .
\end{align*}

\end{cor}

\section{Proofs}

\subsection{Proof of Theorem \ref{mainth}}

\label{proofs}{}

Let $\mathbf{X}_{t}=\left(  X_{t}^{1},\ldots,X_{t}^{N}\right)  $ be a vector
of i.i.d. copies of the exact solution to (\ref{eq:sde}), and define for fixed
$t,$
\[
w_{N}:=\underset{h\in\mathcal{H}_{K_{N}}}{\arg\min}\left\{  \frac{1}{N}%
\sum_{i=1}^{N}\left(  f(X_{T}^{i})-h(X_{t}^{i})\right)  ^{2}\right\}  .
\]
Further let us denote by $V,\widetilde{V}\in\mathbb{R}^{N}$ the column vectors
with coordinates
\[
V_{i}=\frac{f(X_{T}^{i})}{\sqrt{N}},\quad\widetilde{V}_{i}:=\frac
{f(X_{T}^{i,N})}{\sqrt{N}},\quad i=1,\ldots,N,
\]
respectively, and consider the $\mathbb{R}^{N\times K}$ matrices
\begin{align*}
\widetilde{Z}  &  =\left(  \psi_{k}\left(  X_{t}^{i,N}\right)  /\sqrt
{N},\,i=1,\ldots,N,\,k=1,\ldots,K\right)  ,\\
Z  &  =\left(  \psi_{k}\left(  X_{t}^{i}\right)  /\sqrt{N},\,i=1,\ldots
,N,\,k=1,\ldots,K\right)  ,
\end{align*}
respectively. Then we have
\[
\widetilde{w}_{N}=\widetilde{\beta}_{N}^{\top}\boldsymbol{\psi}_{K}\left(
\cdot\right)  ,\quad\widetilde{\beta}_{N}=\left(  \widetilde{Z}^{\top
}\widetilde{Z}\right)  ^{-1}\widetilde{Z}^{\top}\widetilde{V}=\widetilde
{Z}^{\dagger}\widetilde{V}%
\]
and
\[
w_{N}=\beta_{N}^{\top}\boldsymbol{\psi}_{K}\left(  \cdot\right)  ,\quad
\beta_{N}=\left(  Z^{\top}Z\right)  ^{-1}Z^{\top}{V}=Z^{\dagger}V\text{ }%
\]
with $\boldsymbol{\psi}_{K}=\left(  \psi_{1},\ldots,\psi_{K}\right)  ^{\top}.$
By using that
\[
\left\vert T_{M}\widetilde{w}^{N}\left(  x\right)  -T_{M}w^{N}\left(
x\right)  \right\vert \leq\left\vert \widetilde{w}^{N}\left(  x\right)
-w^{N}\left(  x\right)  \right\vert
\]
almost surely, one has for any event $\mathcal{C}\in\mathcal{F}$
\begin{gather*}
\left(  \mathsf{E}\left[  \int\left(  T_{M}\widetilde{w}^{N}\left(  x\right)
-w\left(  x\right)  \right)  ^{2}\mu_{t}(dx)\right]  \right)  ^{1/2}\leq\\
\left(  \mathsf{E}\left[  \int1_{\mathcal{C}}\left(  \widetilde{w}^{N}\left(
x\right)  -w^{N}\left(  x\right)  \right)  ^{2}\mu_{t}(dx)\right]  \right)
^{1/2}+2M_{f}\left(  \mathsf{P}\left[  \Omega\backslash\mathcal{C}\right]
\right)  ^{1/2}\\
+\left(  \mathsf{E}\left[  \int\left(  T_{M}w^{N}\left(  x\right)  -w\left(
x\right)  \right)  ^{2}\mu_{t}(dx)\right]  \right)  ^{1/2}\\
\leq M_{K}\sqrt{K}\left(  \mathsf{E}\left[  \left\Vert \widetilde{\beta}%
_{N}-\beta_{N}\right\Vert ^{2}1_{\mathcal{C}}\right]  \right)  ^{1/2}\\
+2M_{f}\left(  \mathsf{P}\left[  \Omega\backslash\mathcal{C}\right]  \right)
^{1/2}\\
+\left(  \mathsf{E}\left[  \int\left(  T_{M}w^{N}\left(  x\right)  -w\left(
x\right)  \right)  ^{2}\mu_{t}(dx)\right]  \right)  ^{1/2}.
\end{gather*}
Set
\begin{align*}
U  &  =\left(  \psi_{k}(X_{t}^{i}),\,i=1,\ldots,N,\,k=1,\ldots,K\right)
^{\top}\in\mathbb{R}^{N\times K},\\
V  &  =\left(  (f(X_{t}^{i,N})/\sqrt{N},\,i=1,\ldots,N\right)  ^{\top}%
\in\mathbb{R}^{N},\\
E  &  =\left(  (\psi_{k}(X_{t}^{i,N})-\psi_{k}(X_{t}^{i}))/\sqrt
{N},\,i=1,\ldots,N,\,k=1,\ldots,K\right)  \in\mathbb{R}^{N\times K},\\
F  &  =\left(  (f(X_{t}^{i,N})-f(X_{t}^{i}))/\sqrt{N},\,i=1,\ldots,N\right)
^{\top}\in\mathbb{R}^{N},
\end{align*}
then, with $\Sigma=\mathsf{E}\left[  UU^{\top}\right]  ,$ $d=K,$ and
$\left\Vert U_{i}\right\Vert \leq\sqrt{K}M_{K},$ Corollary \ref{cor9} implies%
\[
\left\Vert \widetilde{\beta}^{\circ}-\beta^{\circ}\right\Vert ^{2}\leq
2c_{1}^{2}M_{f}^{2}\left\Vert E\right\Vert ^{2}+2c_{2}^{2}\left\Vert
F\right\Vert ^{2},
\]
on a set $\mathcal{C}$ with probability%
\begin{align*}
\mathsf{P}\left[  \mathcal{C}\right]   &  \geq1-2K\exp\left[  -N\frac
{C\varkappa_{\circ}^{4}}{16KM_{K}^{2}\left(  \varkappa^{\circ}\right)  ^{3}%
}\right] \\
&  -\mathsf{E}[\Vert E\Vert^{p}]\left(  \left(  \frac{\sqrt{5\varkappa^{\circ
}}+1}{\varkappa_{\circ}/2}\right)  ^{p}+1\right)  ,
\end{align*}
where constants $c_{1},$ $c_{2}$ only depend on $\varkappa_{\circ},$
$\varkappa^{\circ}.$ In particular we may take%
\begin{align*}
c_{1}  &  :=\frac{44+8\sqrt{5/\varkappa^{\circ}}}{\varkappa_{\circ}},\text{
\ \ and}\\
c_{2}  &  :=d_{1}+\frac{2\sqrt{5\varkappa^{\circ}}}{3\varkappa_{\circ}}%
=\frac{132+2\sqrt{5\varkappa^{\circ}}+24\sqrt{5/\varkappa^{\circ}}}%
{3\varkappa_{\circ}}.
\end{align*}
As a consequence,
\begin{align*}
\mathsf{E}\left[  \Vert\widetilde{\beta}_{N}-\beta_{N}\Vert^{2}1_{\mathcal{C}%
}\right]   &  \leq2c_{1}^{2}M_{f}^{2}\mathsf{E}\left[  \Vert E\Vert
^{2}\right]  +2c_{2}^{2}\mathsf{E}\left[  \Vert F\Vert^{2}\right] \\
&  \leq2c_{1}^{2}M_{f}^{2}\left(  \mathsf{E}\left[  \frac{1}{N}\sum_{i=1}%
^{N}\sum_{k=1}^{K}\left(  \psi_{k}\left(  X_{t}^{i,N}\right)  -\psi_{k}\left(
X_{t}^{i}\right)  \right)  ^{2}\right]  \right) \\
&  +2c_{2}^{2}\mathsf{E}\left[  \frac{1}{N}\sum_{i=1}^{N}\left(  f\left(
X_{t}^{i,N}\right)  -f\left(  X_{t}^{i}\right)  \right)  ^{2}\right] \\
&  \leq\left(  2c_{1}^{2}M_{f}^{2}\sum_{k=1}^{K}L_{k}^{2}+2c_{2}^{2}L_{f}%
^{2}\right)  \mathsf{E}\left[  \left\vert X_{t}^{\cdot,N}-X_{t}^{\cdot
}\right\vert ^{2}\right]
\end{align*}
We further have for $p\geq2,$
\begin{align*}
\left(  \mathsf{E}\left[  \left\vert E\right\vert ^{p}\right]  \right)
^{1/p}  &  \leq\left(  \mathsf{E}\left[  \left(  \frac{1}{N}\sum_{i=1}^{N}%
\sum_{k=1}^{K}\left(  \psi_{k}\left(  X_{t}^{i,N}\right)  -\psi_{k}\left(
X_{t}^{i}\right)  \right)  ^{2}\right)  ^{p/2}\right]  \right)  ^{1/p}\\
&  =\left(  \left(  \mathsf{E}\left[  \left(  \frac{1}{N}\sum_{i=1}^{N}%
\sum_{k=1}^{K}\left(  \psi_{k}\left(  X_{t}^{i,N}\right)  -\psi_{k}\left(
X_{t}^{i}\right)  \right)  ^{2}\right)  ^{p/2}\right]  \right)  ^{2/p}\right)
^{1/2}\\
&  \leq\left(  \frac{1}{N}\sum_{i=1}^{N}\sum_{k=1}^{K}\left(  \mathsf{E}%
\left[  \left(  \psi_{k}\left(  X_{t}^{i,N}\right)  -\psi_{k}\left(  X_{t}%
^{i}\right)  \right)  ^{p}\right]  \right)  ^{2/p}\right)  ^{1/2}\\
&  \leq\left(  \frac{1}{N}\sum_{i=1}^{N}\sum_{k=1}^{K}L_{k}^{2}\left(
\mathsf{E}\left[  \left\vert X_{t}^{i,N}-X_{t}^{i}\right\vert ^{p}\right]
\right)  ^{2/p}\right)  ^{1/2}\\
&  =\sqrt{\sum_{k=1}^{K}L_{k}^{2}}\left(  \mathsf{E}\left[  \left\vert
X_{t}^{\cdot,N}-X_{t}^{\cdot}\right\vert ^{p}\right]  \right)  ^{1/p}.
\end{align*}
Combining the latter bounds with (\ref{Koh}) and Theorem~11.3 from
\cite{gyorfi2002distribution}, and taking $p=2$ for simplicity, we get (using
subadditivity of the squareroot)%

\begin{gather*}
\left(  \mathsf{E}\left[  \int\left(  T_{M}\widetilde{w}^{N}\left(  x\right)
-w\left(  x\right)  \right)  ^{2}\mu_{t}(dx)\right]  \right)  ^{1/2}\leq\\
\leq M_{K}\sqrt{2K}\left(  c_{1}M_{f}\sqrt{\sum_{k=1}^{K}L_{k}^{2}}+c_{2}%
L_{f}\right)  \frac{C_{2}}{\sqrt{N}}\\
+2M_{f}\sqrt{2K}\exp\left[  -N\frac{C\varkappa_{\circ}^{4}}{32KM_{K}%
^{2}\left(  \varkappa^{\circ}\right)  ^{3}}\right] \\
2M_{f}\sqrt{\sum_{k=1}^{K}L_{k}^{2}}\left(  \frac{\sqrt{5\varkappa^{\circ}}%
+1}{\varkappa_{\circ}/2}+1\right)  \frac{C_{2}}{\sqrt{N}}\\
+c_{3}M_{f}\frac{\sqrt{1+\log N}\sqrt{K}}{\sqrt{N}}\\
+c_{4}\inf_{h\in\mathcal{H}_{K}}\left(  \int\left(  h(x)-w(t,x)\right)
^{2}\mu_{t}(dx)\right)  ^{1/2}%
\end{gather*}
for universal constants $c_{3},c_{4}.$ Summarizing, and using (\ref{lk}),
yields (\ref{res}).

\subsection{Proof of Lemma~\ref{lem23}}

Let us observe that for $j<\mathcal{J},$%
\begin{gather*}
g_{\tau_{j+1}^{\ast}}(Z_{\tau_{j+1}})-g_{\tau_{N,j+1}}(Z_{\tau_{N,j+1}%
})=\left(  g_{j+1}(Z_{j+1})-g_{\tau_{N,j+1}}(Z_{\tau_{N,j+1}})\right)
1_{\{\tau_{j+1}^{\ast}=j+1,\tau_{N,j+1}>j+1\}}\\
+\left(  g_{\tau_{j+1}^{\ast}}(Z_{\tau_{j+1}^{\ast}})-g_{j}(Z_{j})\right)
1_{\{\tau_{j+1}^{\ast}>j+1,\tau_{N,j+1}=j+1\}}\\
+\left(  g_{\tau_{j+1}^{\ast}}(Z_{\tau_{j+1}^{\ast}})-g_{\tau_{N,j+1}}%
(Z_{\tau_{N,j+1}})\right)  1_{\{\tau_{j+1}^{\ast}>j+1,\tau_{N,j+1}>j+1\}}.
\end{gather*}
By abbreviating temporarily in this proof $\mathsf{E:=E}_{\mathcal{G}_{j+1}},$
and denoting $\mathcal{R}_{N,j}:=\mathsf{E}\left[  \left.  g_{\tau_{j+1}%
^{\ast}}(Z_{\tau_{j+1}^{\ast}})-g_{\tau_{N,j+1}}(Z_{\tau_{N,j+1}})\right\vert
Z_{j}\right]  ,$ we have $\mathcal{R}_{N,j}\geq0$ almost surely, and%
\begin{align}
\mathcal{R}_{N,j}  &  =\mathsf{E}\left[  \left.  \left(  g_{j+1}%
(Z_{j+1})-\mathsf{E}\left[  \left.  g_{\tau_{N,j+2}}(Z_{\tau_{N,j+2}%
})\right\vert Z_{j+1}\right]  \right)  1_{\{\tau_{j+1}^{\ast}=j+1,\tau
_{N,j+1}>j+1\}}\right\vert Z_{j}\right] \nonumber\\
&  +\mathsf{E}\left[  \left.  \left(  \mathsf{E}\left[  \left.  g_{\tau
_{j+2}^{\ast}}(Z_{\tau_{j+2}^{\ast}})\right\vert Z_{j+1}\right]
-g_{j+1}(Z_{j+1})\right)  1_{\{\tau_{j+1}^{\ast}>j+1,\tau_{N,j+1}%
=j+1\}}\right\vert Z_{j}\right] \nonumber\\
&  +\mathsf{E}\left[  \left.  \mathsf{E}\left[  \left.  g_{\tau_{j+2}^{\ast}%
}(Z_{\tau_{j+2}^{\ast}})-g_{\tau_{N,j+2}}(Z_{\tau_{N,j+2}})\right\vert
Z_{j+1}\right]  1_{\{\tau_{j+1}^{\ast}>j+1,\tau_{N,j+1}>j+1\}}\right\vert
Z_{j}\right] \nonumber\\
&  =T_{1}+T_{2}+\mathsf{E}\left[  \left.  \mathcal{R}_{N,j+1}1_{\{\tau
_{j+1}^{\ast}>j+1,\tau_{N,j+1}>j+1\}}\right\vert Z_{j}\right]  . \label{c1}%
\end{align}
For $T_{1}$ we have%
\begin{align*}
T_{1}  &  =\mathsf{E}\left[  \left.  \left(  g_{j+1}(Z_{j+1})-\mathsf{E}%
\left[  \left.  g_{\tau_{j+2}^{\ast}}(Z_{\tau_{j+2}^{\ast}})\right\vert
Z_{j+1}\right]  \right)  1_{\{\tau_{j+1}^{\ast}=j+1,\tau_{N,j+1}%
>j+1\}}\right\vert Z_{j}\right] \\
&  +\mathsf{E}\left[  \left.  \left(  \mathsf{E}\left[  \left.  g_{\tau
_{j+2}^{\ast}}(Z_{\tau_{j+2}^{\ast}})\right\vert Z_{j+1}\right]
-\mathsf{E}\left[  \left.  g_{\tau_{N,j+2}}(Z_{\tau_{N,j+2}})\right\vert
Z_{j+1}\right]  \right)  1_{\{\tau_{j+1}^{\ast}=j+1,\tau_{N,j+1}%
>j+1\}}\right\vert Z_{j}\right]  ,
\end{align*}
and since
\begin{align*}
C_{N,j+1}(Z_{j+1})  &  \geq g_{j+1}(Z_{j+1})\geq\mathsf{E}\left[  \left.
g_{\tau_{j+2}^{\ast}}(Z_{\tau_{j+2}^{\ast}})\right\vert Z_{j+1}\right] \\
&  =C_{j+1}^{\ast}(Z_{j+1})\geq\mathsf{E}\left[  \left.  g_{\tau_{N,j+2}%
}(Z_{\tau_{N,j+2}})\right\vert Z_{j+1}\right]
\end{align*}
on $\{\tau_{j+1}^{\ast}=j+1,\tau_{N,j+1}>j+1\},$ we get
\begin{align}
0  &  \leq T_{1}\leq\mathsf{E}\left[  \left.  \left(  C_{N,j+1}(Z_{l+1}%
)-C_{j+1}^{\ast}(Z_{j+1})\right)  1_{\{\tau_{j+1}^{\ast}=j+1,\tau
_{N,j+1}>j+1\}}\right\vert Z_{j}\right] \nonumber\\
&  +\mathsf{E}\left[  \left.  \mathcal{R}_{N,j+1}1_{\{\tau_{j+1}^{\ast
}=j+1,\tau_{N,j+1}>j+1\}}\right\vert Z_{j}\right]  . \label{c2}%
\end{align}
Similarly, for $T_{2}$ we have
\begin{equation}
0\leq T_{2}\leq\mathsf{E}\left[  \left.  \left(  C_{j+1}^{\ast}(Z_{j+1}%
)-C_{N,j+1}(Z_{j+1})\right)  1_{\{\tau_{j+1}^{\ast}>j+1,\tau_{N,j+1}%
=j+1\}}\right\vert Z_{j}\right]  . \label{c3}%
\end{equation}
Combining (\ref{c1}), (\ref{c2}), and (\ref{c3}), yields%
\[
\mathcal{R}_{N,j}\leq\mathsf{E}\left[  \left.  \left\vert C_{N,j+1}%
(Z_{j+1})-C_{j+1}^{\ast}(Z_{j+1})\right\vert \right\vert Z_{j}\right]
+\mathsf{E}\left[  \left.  \mathcal{R}_{N,j+1}\right\vert Z_{j}\right]  .
\]
By straightforward induction, using the tower property and the final condition
$\mathcal{R}_{N,\mathcal{J}-1}=0,$ we so obtain%
\[
0\leq C_{j}^{\ast}\left(  Z_{j}\right)  -\overline{C}_{N,j}\left(
Z_{j}\right)  \leq\sum_{l=j+1}^{\mathcal{J}-1}\mathsf{E}\left[  \left.
\left\vert C_{N,l}(Z_{l})-C_{l}^{\ast}(Z_{l})\right\vert \right\vert
Z_{j}\right]  .
\]
Taking on both sides the $L_{p}$-norm, applying the triangle inequality, and
using that%
\[
\mathsf{E}\left[  \mathsf{E}\left[  \left.  \left\vert C_{N,l}(Z_{j}%
)-C_{l}^{\ast}(Z_{l})\right\vert \right\vert Z_{j}\right]  ^{p}\right]
\leq\mathsf{E}\left[  \left\vert C_{N,l}(Z_{l})-C_{l}^{\ast}(Z_{l})\right\vert
^{p}\right]  ,
\]
finally gives (\ref{eq:bound_1}).

\subsection{Proof of Theorem~\ref{thm_main}}

Theorem \ref{mainth} implies,%
\[
\mathsf{E}_{\mathcal{G}_{j+1}}\left[  \left\Vert C_{N,j}-\overline{C}%
_{N,j}\right\Vert _{L_{2}(\mu_{j})}^{2}\right]  \leq c_{1}^{2}\epsilon
_{N,K}^{2}+c_{2}^{2}\inf_{h\in\mathcal{H}_{K}}\left\Vert \overline{C}%
_{N,j}\left(  \cdot\right)  -h\right\Vert _{L_{2}(\mu_{j})}^{2},
\]
almost surely, for some $\eta,c_{1},c_{2}>0,$ which do not depend on $j,K,$
and $N.$ Hence, for the unconditional expectation we get,%
\[
\mathsf{E}\left[  \left\Vert C_{N,j}-\overline{C}_{N,j}\right\Vert _{L_{2}%
(\mu_{j})}^{2}\right]  \leq c_{1}^{2}\epsilon_{N,K}^{2}+c_{2}^{2}\inf
_{h\in\mathcal{H}_{K}}\mathsf{E}\left[  \left\Vert \overline{C}_{N,j}\left(
\cdot\right)  -h\right\Vert _{L_{2}(\mu_{j})}^{2}\right]
\]
and so
\begin{equation}
\left\Vert C_{N,j}-\overline{C}_{N,j}\right\Vert _{L_{2}(\mu_{j},\mathbb{P}%
)}\leq c_{1}\epsilon_{N,K}+c_{2}\inf_{h\in\mathcal{H}_{K}}\left\Vert
\overline{C}_{N,j}\left(  \cdot\right)  -h\right\Vert _{L_{2}(\mu
_{j},\mathbb{P})} .\label{gy}%
\end{equation}
By using (\ref{gy}) and the unconditional expectation applied to Lemma
\ref{lem23} with $p=2,$ we get%
\begin{align*}
\left\Vert C_{N,j}-C_{j}^{\ast}\right\Vert _{L_{2}(\mu_{j},\mathbb{P})}  &
\leq\left\Vert C_{N,j}-\overline{C}_{N,j}\right\Vert _{L_{2}(\mu
_{j},\mathbb{P})}+\left\Vert \overline{C}_{N,j}-C_{j}^{\ast}\right\Vert
_{L_{2}(\mu_{j},\mathbb{P})}\\
&  \leq c_{1}\epsilon_{N,K}+c_{2}\inf_{h\in\mathcal{H}_{K}}\left\Vert
C_{j}^{\ast}-h\right\Vert _{L_{2}(\mu_{j})}\\
&  +(1+c_{2})\left\Vert \overline{C}_{N,j}-C_{j}^{\ast}\right\Vert _{L_{2}%
(\mu_{j},\mathbb{P})}\\
&  \leq\Delta_{N,K}+(1+c_{2})\sum_{l=j+1}^{\mathcal{J}}\left\Vert
C_{N,l}-C_{l}^{\ast}\right\Vert _{L_{2}(\mu_{l},\mathbb{P})}%
\end{align*}
with $c=\max\{c_{1},c_{2}\}.$ We prove the statement by induction. Suppose
that the inequality \eqref{eq: main} holds for $j=k,$ then
\begin{align*}
\left\Vert C_{N,\mathcal{J}-k-1}-C_{\mathcal{J}-k-1}^{\ast}\right\Vert
_{L_{2}(\mu_{\mathcal{J}-k-1},\mathbb{P})}  &  \leq\Delta_{N,K}\\
&  +(1+c_{2})\sum_{l=0}^{k}\left\Vert C_{N,\mathcal{J}-l}-C_{\mathcal{J}%
-l}^{\ast}\right\Vert _{L_{2}(\mu_{\mathcal{J}-l},\mathbb{P})},\\
&  \leq\Delta_{N,K}+\Delta_{N,K}(1+c)\sum_{l=0}^{k}(2+c)^{l}\\
&  =\Delta_{N,K}(1+((2+c)^{k+1}-1))\\
&  =\Delta_{N,K}(2+c)^{k+1}%
\end{align*}
and \eqref{eq: main} holds also for $j=k+1.$

\section{Appendix}

In this section we present two auxiliary lemmas that were needed in
Section~\ref{pertan}.

\begin{lem}
\label{App} Let $\rho>0$ and the matrix $Z\in\mathbb{R}^{N\times d}$ be of
full rank with $N>d.$ Let $Z$ and $E\in\mathbb{R}^{N\times d}$ be such that
\begin{equation}
\lambda_{\mathrm{min}}\left(  Z^{\top}Z\right)  -\left(  2\Vert Z\Vert
+1\right)  \Vert E\Vert>\rho,\text{ \ \ }\Vert E\Vert<1. \label{cond}%
\end{equation}
Then we have
\begin{equation}
\Vert(Z+E)^{\dagger}-Z^{\dagger}\Vert\leq\frac{\Vert E\Vert}{\rho}\left[
1+\frac{\left(  2\Vert Z\Vert+1\right)  \Vert Z\Vert}{\rho}\right]  .
\label{onA}%
\end{equation}

\end{lem}

\begin{proof}
Denote
\[
\Delta=Z^{\top}E+E^{\top}Z+E^{\top}E,
\]
then using the identity
\begin{align*}
\left(  (Z+E)^{\top}(Z+E)\right)  ^{-1}-\left(  Z^{\top}Z\right)  ^{-1}  &
=-\left(  (Z+E)^{\top}(Z+E)\right)  ^{-1}\Delta\left(  Z^{\top}Z\right)
^{-1}\\
&  =-\left(  Z^{\top}Z+\Delta\right)  ^{-1}\Delta\left(  Z^{\top}Z\right)
^{-1},
\end{align*}
we derive%
\begin{multline}
\Vert(((Z+E)^{\top}(Z+E))^{-1}-(Z^{\top}Z)^{-1})Z^{\top}\Vert\\
\leq\Vert(Z^{\top}Z+\Delta)^{-1}\Vert\Vert(Z^{\top}Z)^{-1}\Vert(2\Vert
Z\Vert+1)\Vert E\Vert\Vert Z\Vert\\
\leq\frac{(2\Vert Z\Vert+1)\Vert E\Vert\Vert Z\Vert}{\rho^{2}}, \label{onA11}%
\end{multline}
since we have $\Vert(Z^{\top}Z)^{-1}\Vert=\lambda_{\mathrm{min}}^{-1}\left(
Z^{\top}Z\right)  <\rho^{-1}$ and
\begin{align*}
\lambda_{\mathrm{min}}\left(  Z^{\top}Z+\Delta\right)   &  =\underset
{\left\vert x\right\vert =1}{\inf}x^{\top}\left(  Z^{\top}Z+\Delta\right)
x\geq\underset{\left\vert x\right\vert =1}{\inf}x^{\top}Z^{\top}%
Zx+\underset{\left\vert x\right\vert =1}{\inf}x^{\top}\Delta x\\
&  \geq\lambda_{\mathrm{min}}\left(  Z^{\top}Z\right)  -\Vert\Delta\Vert
\geq\lambda_{\mathrm{min}}\left(  Z^{\top}Z\right)  -\left(  2\Vert
Z\Vert+1\right)  \Vert E\Vert\\
&  >\rho>0.
\end{align*}
Analogously we have
\begin{equation}
\left\Vert \left(  (Z+E)^{\top}(Z+E)\right)  ^{-1}E\right\Vert =\left\Vert
\left(  Z^{\top}Z+\Delta\right)  ^{-1}E\right\Vert \leq\frac{\left\Vert
E\right\Vert }{\rho}, \label{onA2}%
\end{equation}
and then (\ref{onA}) follows by (\ref{onA11}), (\ref{onA2}), and the triangle inequality.
\end{proof}

\begin{lem}
\label{lem:spectrum_subgaussian} Let $X_{1},\ldots,X_{N}$ be independent
identically distributed random vectors in $\mathbb{R}^{d}$ such that%
\[
\mathsf{E}\left[  X_{1}X_{1}^{\top}\right]  =\Sigma
\]
and for some $M>0,$ $\Vert X_{1}\Vert\leq M$ almost surely. Then for all
$\delta\in(0,1),$
\begin{multline}
\mathbb{P}\left(  \left\{  \lambda_{\max}\left(  \frac{1}{N}\sum_{i=1}%
^{N}X_{i}X_{i}^{\top}\right)  >\lambda_{\max}\left(  \Sigma\right)
+\varepsilon_{\delta,N}\right\}  \cup\right. \label{lemed}\\
\left.  \left\{  \lambda_{\min}\left(  \frac{1}{N}\sum_{i=1}^{N}X_{i}%
X_{i}^{\top}\right)  <\lambda_{\min}\left(  \Sigma\right)  -\varepsilon
_{\delta,N}\right\}  \right)  \leq\delta,
\end{multline}
where
\begin{equation}
\varepsilon_{\delta,N}=M\sqrt{\frac{\log(2d/\delta)}{NC}}\frac{\lambda_{\max
}^{3/2}(\Sigma)}{\lambda_{\min}(\Sigma)} \label{epdel}%
\end{equation}
for some absolute constant $C>0,$ provided $N$ is large enough such that
\begin{equation}
M\sqrt{\frac{\log(2d/\delta)}{NC}}\leq\lambda_{\max}^{1/2}(\Sigma).
\label{epdelc}%
\end{equation}

\end{lem}

\begin{proof}
For $z>0$ we have%
\begin{gather}
\mathrm{P}\left[  \lambda_{\max}\left(  \frac{1}{N}\sum_{i=1}^{N}X_{i}%
X_{i}^{\top}\right)  -\lambda_{\max}\left(  \Sigma\right)  >z\right]
\label{oth1}\\
\leq\mathrm{P}\left[  \left\Vert \frac{1}{N}\sum_{i=1}^{N}X_{i}X_{i}^{\top
}-\Sigma\right\Vert >z\right]  .\nonumber
\end{gather}
On the other hand, since for positive matrices $A,B,$%
\[
\left\vert \lambda_{\min}\left(  A\right)  -\lambda_{\min}\left(  B\right)
\right\vert \leq\left\Vert B^{-1}\right\Vert \left\Vert B\right\Vert
\left\Vert B-A\right\Vert ,
\]
we have for $0<z<\lambda_{\min}\left(  \Sigma\right)  ,$%
\begin{align}
&  \mathrm{P}\left[  \lambda_{\min}\left(  \frac{1}{N}\sum_{i=1}^{N}X_{i}%
X_{i}^{\top}\right)  <\lambda_{\min}\left(  \Sigma\right)  -z\right]
\label{oth}\\
&  \leq\mathrm{P}\left[  \left\Vert \frac{1}{N}\sum_{i=1}^{N}X_{i}X_{i}^{\top
}-\Sigma\right\Vert >z\frac{\lambda_{\min}(\Sigma)}{\lambda_{\max}(\Sigma
)}\right]  .\nonumber
\end{align}
Theorem~5.44 in \cite{MR2963170} implies that for any $s>0,$
\[
\mathrm{P}\left[  \left\Vert \frac{1}{N}\sum_{i=1}^{N}X_{i}X_{i}^{\top}%
-\Sigma\right\Vert >\max\left\{  \Vert\Sigma\Vert^{1/2}\sqrt{\frac{s^{2}M^{2}%
}{N}},\frac{s^{2}M^{2}}{N}\right\}  \right]  \leq d\cdot\exp(-Cs^{2}),
\]
where $C$ is an absolute constant. For $N$ such that $s^{2}M/N\leq\Vert
\Sigma\Vert$ and $s=\sqrt{C^{-1}\log(2d/\delta)},$ we so obtain
\begin{gather}
\mathrm{P}\left[  \left\Vert \frac{1}{N}\sum_{i=1}^{N}X_{i}X_{i}^{\top}%
-\Sigma\right\Vert >M\sqrt{\frac{\Vert\Sigma\Vert\log(2d/\delta)}{NC}}\right]
\leq\delta/2,\text{ \ \ for }N\text{ such that}\label{oth2}\\
M\sqrt{\frac{\log(2d/\delta)}{NC}}\leq\Vert\Sigma\Vert^{1/2}=\lambda_{\max
}^{1/2}\left(  \Sigma\right)  .\nonumber
\end{gather}
Thus, (\ref{lemed}) follows from (\ref{oth2}) and taking $z=$ $\varepsilon
_{\delta,N}$ given by (\ref{epdel}) in (\ref{oth1}) and (\ref{oth}), respectively.
\end{proof}

\bibliographystyle{plain}
\bibliography{perturbation_analysis,particles,MV_optimal_stop}

\end{document}